\documentclass[a4paper,reqno]{amsart}
\usepackage{amsmath,amsfonts,amscd,amssymb,graphicx}
\usepackage{ifthen} % erforderlich fŸr L^p-Normen
\allowdisplaybreaks[4]
\numberwithin{figure}{section}
\numberwithin{table}{section}
\numberwithin{equation}{section}
\newcommand{\rz}{\mathbb{R}}       % reelle Zahlen
\theoremstyle{plain}
\newtheorem{SATZ}{Proposition}[section]   % Numerierung c.s.x
\newtheorem{THM}[SATZ]{Theorem}
\newtheorem{LEM}[SATZ]{Lemma}

\theoremstyle{definition}

\newtheorem{BEM}[SATZ]{Remark}

\DeclareMathOperator{\diam}{diam}
\DeclareMathOperator{\Div}{div}
\newcommand{\dx}{{\rm d}}
\newcommand{\halb}{\frac 12}
\DeclareMathOperator{\shp}{C}
\newcommand{\shape}[1][\CT]{\shp_{#1}}
\newcommand{\hdiamE}[1][E]{h_{#1}}
\newcommand{\hdiamK}[1][K]{h_{#1}}
\newcommand{\be}{\beta}
\newcommand{\ga}{\gamma}
\newcommand{\Ga}{\Gamma}

\newcommand{\De}{\Delta}
\newcommand{\ep}{\varepsilon}
\newcommand{\ka}{\kappa}
\newcommand{\la}{\lambda}

\newcommand{\om}{\omega}
\newcommand{\Om}{\Omega}
\newcommand{\vfi}{\varphi}
\newcommand{\si}{\sigma}

\newcommand\CE{\mathcal{E}}
\newcommand\CI{\mathcal{I}}
\newcommand\CM{\mathcal{M}}

\newcommand\CT{\mathcal{T}}
\DeclareMathOperator{\const}{c}
\newcommand{\cemb}{\const_p}
\newcommand{\creac}{\const_b}
\newcommand{\cfr}{\const_F}
\newcommand{\betrag}[2][]{\left\lvert #2 \right\rvert_{#1}} % Betrag
\newcommand{\norm}[1]{\left\lVert #1 \right\rVert} % Norm
\newcommand{\normL}[3][2]{
	\ifthenelse{\equal{#1}{2}}
	{\ifthenelse{\equal{#2}{\Om}}
		{\left\lVert #3 \right\rVert}
		{\left\lVert #3 \right\rVert_{#2}}
	}
	{\left\lVert #3 \right\rVert_{L^{#1}(#2)}}
	} % Lp-Norm
\newcommand{\normE}[2][]{\left\lVert{\hskip -0.25em} \left\lvert #2 \right\rVert {\hskip -0.25em} \right\rvert_{#1}} % triple bar Norm
\newcommand{\normD}[2][]{\left\lVert{\hskip -0.25em} \left\lvert #2 
\right\rVert 
{\hskip -0.25em} \right\rvert_{\ast #1}} % triple bar dual Norm
\newcommand{\normX}[1]{\left\lVert #1 \right\rVert_X} % Norm X
\newcommand{\normXD}[3][0]{\left\lVert #3 \right\rVert_{L^2(#1,#2;H^{-1}(\Om))}} % Norm X^*
\newcommand{\paar}[2]{\left\langle #1\,,\,#2 \right\rangle} % Paarung
\newcommand\Va{\mathbf{a}}
\newcommand{\cd}{\Va \cdot \nabla}
\newcommand{\cdm}{\bar{\Va}_{\CM_n}\cdot\nabla}
\newcommand{\cdt}{\bar{\Va}_{\CT_n}\cdot\nabla}
\newcommand\Vn{\mathbf{n}}
\newcommand{\jump}[2][E]{{\mathbb J}_{#1}(#2)}
\newcommand{\mdiam}[1]{\hslash_{#1}}
\title[Non-Stationary Non-Linear Convection-Diffusion Equations]{A Posteriori Error Estimates for Non-Stationary Non-Linear Convection-Diffusion Equations}
\author{R. Verf{\"u}rth}
\address{Ruhr-Uni\-ver\-si\-t{\"a}t Bo\-chum \\ Fa\-kul\-t{\"a}t f{\"u}r Ma\-the\-ma\-tik \\ D-44780 Bo\-chum \\ Germany}
\email{ruediger.verfuerth@ruhr-uni-bochum.de}
\date{\today}
\date{\today}
\begin{document}
\begin{abstract}
Motivated by stochastic convection-diffusion problems we derive a posteriori error estimates for non-stationary non-linear convection-diffusion equations acting as a deterministic paradigm. The problem considered here neither fits into the standard linear framework due to its non-linearity nor into the standard non-linear framework due to the lacking differentiability of the non-linearity. Particular attention is paid to the interplay of the various parameters controlling the relative sizes of diffusion, convection, reaction, and non-linearity (noise).
\end{abstract}
\maketitle
%
% insert todo-list
%
%\makeatletter \providecommand\@dotsep{5} \makeatother {\color{blue}\listoftodos\relax}
%
% Introduction
%
\section{Introduction}\label{S:intro}
Recently stochastic convection-diffusion problems have attracted considerable interest \cite{BanBrzNekPro,CarMulPro12,DunHauPro12,DunPro16,HutJen,LiuRoe,WeiHutJenKru}. To obtain efficient numerical discretizations adaptivity is mandatory. Yet, for these problems, adaptivity in general and a posteriori error estimates in particular are still in their infancy. As a first step to close this gap we consider in this article deterministic non-stationary convection-diffusion equations with a non-linearity of the form $\nu \vfi(u) g$ modelling the noise (cf. equation \eqref{E:nonlin-conv-diff} below). They neither fit into the framework of \cite[\S 3]{TobVer15} and \cite[\S 6.2]{Ver13} due to the non-linearity, nor into the framework of \cite{AmrWih,GeoLakVir}, \cite{Ver98d, Ver98e}, and \cite[\S 6.6]{Ver13} due to the lacking differentiability of the non-linearity or its lacking strong monotonicity. Therefore, in what follows, we will carefully adapt the arguments of \cite[\S 3]{TobVer15} and \cite[\S 6.2]{Ver13} to catch the interplay of the various parameters controlling the relative size of diffusion, convection, reaction, and non-linearity.

The article is organized as follows. In Section \ref{S:varprob} we present the differential equation and its variational formulation. Section  \ref{S:discprob} gives the discretization which is a stabilized $\theta$-scheme with a possibly explicit treatment of the non-linearity. In Section \ref{S:apost} we then derive the a posteriori error estimates (cf. Theorem \ref{T:apost}).
%
% Variational Problem
%
\section{Variational Problem}\label{S:varprob}
As a deterministic paradigm for stochastic convection-diffusion problems, we consider the following non-stationary non-linear convection-diffusion equations:
\begin{equation}\label{E:nonlin-conv-diff}
\begin{aligned}
\partial_t u - \ep \De u + \cd u + b u &= \nu \vfi(u) g &&\text{in } \Om \times (0,T], \\
u &= 0 &&\text{on } \Ga \times (0,T], \\
u(\cdot,0) &= u_0 &&\text{in } \Om.
\end{aligned}
\end{equation}
Here, $\Om \subset \rz^d$, $d \in \{2, 3\}$, is a bounded polyhedral domain with Lipschitz boundary $\Ga$. The final time $T$ is arbitrary, but kept fixed in what follows. We assume that the data satisfy the following conditions (compare \cite[\S 3]{TobVer15} and \cite[\S 6.2]{Ver13}):
\begin{enumerate}
\renewcommand{\labelenumi}{(A\arabic{enumi})}
\item $\ep > 0$, $\nu \ge 0$,
\item $g \in L^\infty(\Om \times (0,T])$, $\Va \in C(0,T;W^{1,\infty}(\Om)^d)$, $b \in L^\infty(\Om \times (0,T])$, $u_0 \in L^2(\Om)$,
\item there are two constants $\be \ge 0$ and $\creac \ge 0$, which do not depend on $\ep$, such that $-\halb \Div \Va + b \ge \be$ in $\Om \times (0,T]$ and $\normL[\infty]{\Om \times (0,T]}{b} \le \creac \be$,
\item the function $\vfi$, modelling the noise, is Lipschitz continuous, i.e. \\
$\betrag{\vfi(s_1) - \vfi(s_2)} \le L \betrag{s_1 - s_2}$ for all $s_1, s_2 \in \rz$.
\end{enumerate}
Examples of functions satisfying assumption (A4) with $L = 1$ are $\vfi(s) = 1 + \betrag{s}$ and $\vfi(s) = \sqrt{1 + s^2}$.
\\ 
We will be particularly interested in the convection-dominated regime $\ep \ll 1$. At the expense of more technical arguments and additional data oscillations, the second assumption can be replaced by slightly weaker conditions concerning the temporal regularity. The third assumption allows us to simultaneously handle the case of a non-vanishing reaction term and the one of absent reaction. If $b \ne 0$ we may assume without loss of generality that $\creac \ge 1$; if $b = 0$ we set $\be = 0$ and $\creac = 1$.

We denote by $L^p(\Om)$ and $W^{k,p}(\Om)$, $1 \le p \le \infty$, $k \ge 1$, the standard Lebesgue and Sobolev spaces equipped with their standard norms $\normL[p]{\Om}{\cdot}$ and $\norm{\cdot}_{W^{k,p}(\Om)}$ respectively, by $H^1_0(\Om)$ the space of all functions in $W^{1,2}(\Om)$ with vanishing trace, and by $H^{-1}(\Om)$ the dual space of $H^1_0(\Om)$. The norms of $H^1_0(\Om)$ and $H^{-1}(\Om)$ depend on the parameters $\ep$ and $\be$ and are specified in \eqref{E:Enorm} and \eqref{E:Dnorm} below. Further, we define a bilinear form $B : H^1_0(\Om) \times H^1_0(\Om) \rightarrow \rz$ and a non-linear map $N : H^1_0(\Om) \rightarrow H^{-1}(\Om)$ by setting for all $u, v \in H^1_0(\Om)$
\begin{equation}\label{E:defB-N}
\begin{split}
B(u,v) &= \int_\Om \left( \ep \nabla u \cdot \nabla v + \cd u v + b u v \right), \\
\paar{N(u)}{v} &= \int_\Om \nu \vfi(u) g v.
\end{split}
\end{equation}
Remind that $B$ and $N$ depend on time $t$ due to the functions $\Va$, $b$, and $g$.

The variational formulation of problem \eqref{E:nonlin-conv-diff} then is to find a function $u$ in $L^2(0,T;H^1_0(\Om))$ with its weak temporal derivative $\partial_t u$ in $L^2(0,T;H^{-1}(\Om))$ such that $u(\cdot,0) = u_0$ almost everywhere and
\begin{equation}\label{E:varprob}
\paar{\partial_t u}{v} + B(u,v) = \paar{N(u)}{v}
\end{equation}
for all $v \in H^1_0(\Om)$ and almost all $t \in (0,T)$.

In what follows we assume that problem \eqref{E:varprob} admits at least one solution.
%
% Discretization
%
\section{Discrete Problem}\label{S:discprob}
For the space-time discretization of problem \eqref{E:nonlin-conv-diff}, we consider partitions $\CI = \left\{ [t_{n-1},t_n] : 1 \le n \le N_\CI \right\}$ of the time-interval $[0,T]$ into sub-intervals satisfying $0 = t_0 < \ldots < t_{N_\CI} = T$. For every $n$ with $1 \le n \le N_\CI$, we denote by $I_n = [t_{n-1},t_n]$ the $n$-th sub-interval and by $\tau_n = t_n - t_{n-1}$ its length. With every intermediate time $t_n$, $0 \le n \le N_\CI$, we associate a partition $\CT_n$ of $\Om$ and a corresponding finite element space $V(\CT_n)$. The partitions $\CI$ and $\CT_n$ and the spaces $V(\CT_n)$ must satisfy the following assumptions (compare \cite[\S 3]{TobVer15} and \cite[\S 6.2]{Ver13}):
\begin{itemize}
\item The closure of $\Om$ is the union of all elements in $\CT_n$.
\item Every element has at least one vertex in $\Om$.
\item Every element in $\CT_n$ is either a simplex or a parallelepiped, i.e.\ it is the image of the $d$-dimensional reference simplex $\widehat K_d = \left\{ x\in \rz^d : x_1 \ge 0, \, \ldots, \, x_d \ge 0, \right.$\linebreak[4]$\left.x_1 + \ldots + x_d \le 1 \right\}$ or of the $d$-dimensional reference cube $\widehat K_d = [0, 1]^d$ under an affine mapping (\emph{affine-equivalence}).
\item Any two elements in $\CT_n$ are either disjoint or share a complete lower dimensional face of their boundaries (\emph{admissibility}).
\item Denoting by $\hdiamK$ the diameter of any element $K$ and by $\rho_K$ the diameter of the largest ball inscribed into $K$, the shape parameter
\begin{equation*}
\shape = \max_{1 \le n \le N_\CI} \max_{K \in \CT_n} \frac{\hdiamK}{\rho_K}
\end{equation*}
is of moderate size independently of $\ep$, $\be$, and $\nu$ (\emph{shape-regularity}).
\item For every $n$ with $1 \le n \le N_\CI$ there is an affine-equivalent, admissible, and shape-regular partition $\widetilde\CT_n$ such that it is a refinement of both $\CT_n$ and $\CT_{n-1}$ and such that
\begin{equation*}
\shape[\widetilde \CT,\CT] = \max_{\phantom{\widetilde \CT_n}1 \le n \le N_\CI\phantom{\widetilde \CT_n}} \max_{K \in \widetilde \CT_n}
\max_{\phantom{\widetilde \CT_n}K^\prime \in \CT_n; K \subset K^\prime\phantom{\widetilde \CT_n}} \frac{\hdiamK[K^\prime]}{\hdiamK}
\end{equation*}
is of moderate size independently of $\ep$, $\be$, and $\nu$ (\emph{transition condition}).
\item Each $V(\CT_n)$ consists of continuous functions which are piecewise polynomials, the degrees being at least one and being bounded uniformly with respect to all partitions $\CT_n$ and $\CI$ (\emph{degree condition}).
\end{itemize}
The transition condition is due to the simultaneous presence of finite element functions defined on different grids. Usually the partition $\CT_n$ is obtained from $\CT_{n-1}$ by a combination of refinement and of coarsening. In this case the transition condition only restricts the coarsening: it should not be too abrupt nor too strong.
\\
The lower bound on the polynomial degrees is needed for the construction of suitable quasi-interpolation operators. The upper bound ensures that the constants in inverse estimates are uniformly bounded.
\\
Notice that we do not impose any shape-condition of the form $\max_n \tau_n \le c \min_n \tau_n$.

For any parameter $\Theta \in [0,1]$ we set for abbreviation
\begin{equation}\label{E:ntheta}
\begin{split}
g^{n\Theta} &= \Theta g(\cdot,t_n) + (1 - \Theta) g(\cdot,t_{n-1}), \\
\Va^{n\Theta} &= \Theta \Va(\cdot,t_n) + (1 - \Theta) \Va(\cdot,t_{n-1}), \\
b^{n\Theta} &= \Theta b(\cdot,t_n) + (1 - \Theta) b(\cdot,t_{n-1})
\end{split}
\end{equation}
and
\begin{equation*}
\begin{split}
B^{n\Theta}(u, v) &= \int_\Om \left\{ \ep \nabla u \cdot \nabla v + \Va^{n\Theta} \cdot \nabla u v + b^{n\Theta} u v \right\}, \\
\paar{N^{n\Theta}(u)}{v} &= \int_\Om \nu \vfi(u) g^{n\Theta} v.
\end{split}
\end{equation*}

For the finite element discretization of problem \eqref{E:nonlin-conv-diff} we consider a stabilized $\theta$-scheme with a possibly explicit treatment of the non-linearity. More precisely we choose two parameters $\theta, \vartheta \in [0,1]$ and look for a sequence $u^n_{\CT_n} \in V(\CT_n)$, $0 \le n \le N_\CI$, such that $u^0_{\CT_0}$ is the $L^2$-projection of $u_0$ onto $V(\CT_0)$ and such that, for $n = 1, \ldots, N_\CI$ and $U^{n\Theta} = \Theta u^n_{\CT_n} + (1 - \Theta) u^{n-1}_{\CT_{n-1}}$, $\Theta \in \{\theta, \vartheta\}$,
\begin{equation}\label{E:discprob}
\int_\Om \frac{1}{\tau_n} (u^n_{\CT_n} - u^{n-1}_{\CT_{n-1}}) v_{\CT_n} + B^{n\theta}(U^{n\theta},v_{\CT_n}) + S^n(U^{n\theta},v_{\CT_n}) = \paar{N^{n\vartheta}(U^{n\vartheta})}{v_{\CT_n}}
\end{equation}
holds for all $v_{\CT_n} \in V(\CT_n)$.

Note that by choosing $\vartheta \ne \theta$ we may handle the non-linear and linear terms in \eqref{E:nonlin-conv-diff} differently. In particular we may choose $\vartheta = 0$ and $\theta \in \{ \halb, 1 \}$ thus using an explicit discretization for the non-linear term and an implicit one for the linear terms.

The term $S^n$ specifies the particular stabilization. It is supposed to be linear in its second argument and affine in its first argument. Note that $S^n$ may contain contributions of the data $g$. Of course, the choice $S^n = 0$ is also possible and corresponds to a standard finite element method without stabilization. Some popular choices of $S^n$ are as follows (cf. \cite{TobVer15} for more details and references):
\begin{itemize}
\item \emph{Streamline diffusion method:} Here, the stabilizing term has the form
\begin{equation*}
\begin{split}
\quad\quad\quad S^n(u,v) = \sum_{K \in \CT_n} \vartheta_K \int_K &\left\{-\ep \Delta u + \Va^{n\theta} \cdot \nabla u + b^{n\theta} u - \nu \vfi(u) g^{n\theta} \right\} \Va^{n\theta} \cdot \nabla v
\end{split}
\end{equation*}
with $\vartheta_K \normL[\infty]{K}{\Va} \le c_S \hdiamK$ for all $K \in \CT_n$ (cf. eg. \cite{HugBro79,RooStyTob08}).
\item \emph{Local projection scheme:} Denoting by $\CM_n$ a macro-partition such that every element in $\CM_n$ is the union of elements in $\CT_n$ and by $I - \ka_{\CM_n}$ the $L^2$-projection onto an appropriate discontinuous projection space $D(\CM_n)$ living on the partition $\CM_n$ and by $\bar{\Va}_{\CM_n}$ a piecewise constant approximation of $\Va^{n\theta}$ on $\CM_n$, we either have
\begin{equation*}
S^n(u,v) = \sum_{M \in \CM_n} \vartheta_M \int_M \ka_{\CM_n} \left( \cdm u \right)  \ka_{\CM_n} \left( \cdm v \right)
\end{equation*}
with $\vartheta_M \normL[\infty]{M}{\Va}\le c_S \hdiamK[M]$ for all $M \in \CM_n$ or
\begin{equation*}
S^n(u,v) = \sum_{M \in \CM_n} \vartheta_M \int_M \ka_{\CM_n} \left( \nabla u \right)  \ka_{\CM_n} \left( \nabla v \right)
\end{equation*}
with $\vartheta_M  \le c_S \normL[\infty]{M}{\Va} \hdiamK[M]$ for all $M \in \CM_n$ (cf. eg. \cite{HeTob12,KnoTob11,TobWin10}).
\item \emph{Subgrid scale approach:} Decomposing the solution space $V(\CT_n)$ into a space of resolvable scales $X(\CT_n)$ and 
a space of unresolvable scales $Y(\CT_n)$ such that $V(\CT_n) = X(\CT_n) \oplus Y(\CT_n)$ and denoting by $\Pi_n: V(\CT_n) \rightarrow Y(\CT_n)$ a projection operator with 
$X(\CT_n) = \ker (\Pi_n)$, we either have
\begin{equation*}
S^n(u,v) = \sum_{K \in \CT_n} \vartheta_K \int_K \left( \cdt \Pi_n(u) \right) \left( \cdt \Pi_n(v) \right)
\end{equation*}
with $\vartheta_K \normL[\infty]{K}{\Va}\le c_S \hdiamK[K]$ for all $K \in \CT_n$ or 
\begin{equation*}
S^n(u,v) = \sum_{K \in \CT_n} \vartheta_K \int_K  \nabla\Pi_n \left(u \right) \nabla\Pi_n \left(v_\CT \right)
\end{equation*}
with $\vartheta_K  \le c_S \normL[\infty]{K}{\Va} \hdiamK[K]$ for all $K \in \CT_n$ (cf. eg. \cite{ErnGue04,Gue99,Gue01,RooStyTob08}).
\item \emph{Continuous interior penalty method:} Denoting by $\CE_{n,\Om}$ the collection of all element faces of $\CT_n$ inside $\Om$ and by $\jump{\cdot}$ the jump across such a face, we have 
\begin{equation*}
S^n(u,v) = \sum_{E \in \CE_{n,\Om}} \vartheta_E \int_E \jump{\Va^{n\theta} \cdot \nabla u}  \jump{\Va^{n\theta} \cdot \nabla v}
\end{equation*}
with $\vartheta_E \le c_S \hdiamE^2$ for all $E \in \CE_{n,\Om}$ (cf. eg. \cite{ElAErnBur07,Bur05,BurErn07,BurHan04,ErnGue13}).
\end{itemize}

In what follows we assume that problem \eqref{E:discprob} admits at least one solution.
%
% A posteriori error estimates
%
\section{A Posteriori Error Estimates}\label{S:apost}
In what follows we consider a solution $u$ of the variational problem \eqref{E:varprob} and a solution $\left( u^n_{\CT_n} \right)_{0 \le n \le N_\CI}$ of the discrete problem \eqref{E:discprob}. With the latter we associate the function $u_\CI$ which is continuous and piecewise affine with respect to time and which equals $u^n_{\CT_n}$ at time $t_n$, $0 \le n \le N_\CI$. We want to derive explicitly computable a posteriori error estimates which yield upper and lower bounds  for the error $u - u_\CI$. In doing so we pay particular attention to the dependence of the bounds on the parameters $\ep$, $\be$, and $\nu$. To this end we proceed as in \cite{TobVer15} and \cite[\S 6.2]{Ver13}:
\begin{itemize}
\item We introduce the residual associated with the error and prove that a suitable norm of the error is bounded from below and above by a suitable dual norm of the residual.
\item We additively split the residual into three contributions called data residual, temporal residual, and spatial residual.
\item We separately bound the dual norms of the data, temporal, and spatial residuals. 
\end{itemize}
In following this path, we must pay particular attention to the non-linearity. Its Lipschitz-continuity will be crucial.

\subsection{Norms}\label{S:norms} We equip $H^1_0(\Om)$ with the energy norm
\begin{equation}\label{E:Enorm}
\normE{v}  = \left\{ \ep \normL{\Om}{\nabla v}^2 + \be \normL{\Om}{v}^2 \right\}^\halb
\end{equation}
and $H^{-1}(\Om)$ by the corresponding dual norm
\begin{equation}\label{E:Dnorm}
\normD{\ell} = \sup_{v \in H^1_0(\Om) \setminus \{ 0 \}} \frac{\paar{\ell}{v}}{\normE{v}},
\end{equation}
where $\normL{\om}{\cdot}$ is the standard $L^2$-norm on any measurable subset $\om$ of $\Om$ and $\normL{\Om}{\cdot} = \norm{\cdot}_\Om$.
\\
For abbreviation we set for $0 \le t_- < t_+ \le T$
\begin{equation*}
X(t_-,t_+) = L^2(t_-,t_+;H^1_0(\Om)) \cap L^\infty(t_-,t_+;L^2(\Om)) \cap H^1(t_-,t_+;H^{-1}(\Om)),
\end{equation*}
equip it with the norm
\begin{equation*}
\begin{split}
\norm{u}_{X(t_-,t_+)} &= \left\{ \sup_{t_- < t < t_+} \normL{\Om}{u(\cdot,t)}^2 + \int_{t_-}^{t_+} \normE{u(\cdot,t)}^2 \right. \\
&\quad\quad\left.+ \int_{t_-}^{t_+} \normD{\partial_t u(\cdot,t) + \cd u(\cdot,t)}^2 \right\}^\halb,
\end{split}
\end{equation*}
and set
\begin{equation*}
X = X(0,T), \quad \normX{\cdot} = \norm{\cdot}_{X(0,T)}.
\end{equation*}
Recall that for $0 \le t_- < t_+ \le T$ and $\ell : (t_-,t_+) \rightarrow H^{-1}(\Om)$
\begin{equation*}
\normXD[t_-]{t_+}{\ell} = \left\{ \int_{t_-}^{t_+} \normD{\ell(t)}^2 \right\}^\halb.
\end{equation*}
Denote by
\begin{equation}\label{E:Fried}
\cfr = \sup_{v \in H^1_0(\Om) \setminus \{ 0 \}} \frac{\normL{\Om}{v}}{\normL{\Om}{\nabla v}}
\end{equation}
the best constant in Friedrich's inequality. Note that $\cfr \lesssim \diam (\Om)$. Setting
\begin{equation*}
\la = \min \left\{ \cfr \ep^{-\halb}, \be^{-\halb} \right\}
\end{equation*}
equations \eqref{E:Enorm} and \eqref{E:Fried} imply for every $v \in H^1_0(\Om)$
\begin{equation}\label{E:L2Enorm}
\normL{\Om}{v} \le \la \normE{v}.
\end{equation}
For abbreviation we finally set
\begin{equation*}
\ga(t) = \normL[\infty]{\Om}{g(\cdot,t)}, \quad \ga = \normL[\infty]{\Om \times (0,T)}{g}.
\end{equation*}

\subsection{Lipschitz-continuity of the non-linearity}\label{S:contN} The non-linearity $N$ is not differentiable, but Lipschitz-continuous.

\begin{LEM}[Lipschitz-continuity of $N$]\label{L:LipN}
For every $t \in (0,T)$ and $u_1, u_2, v \in \linebreak[4]H^1_0(\Om)$ we have
\begin{equation*}
\paar{N(u_1) - N(u_2)}{v} \le \nu L \ga(t) \normL{\Om}{u_1 - u_2} \normL{\Om}{v}
\end{equation*}
and
\begin{equation*}
\begin{split}
\normD{N(u_1) - N(u_2)} &\le \nu L \la \ga(t) \normL{\Om}{u_1 - u_2} \\
&\le \nu L \la^2 \ga(t) \normE{u_1 - u_2}.
\end{split}
\end{equation*}
\end{LEM}

\begin{proof}
For every $v \in H^1_0(\Om)$ and $t \in (0,T)$ we have thanks to assumption (A4)
\begin{equation*}
\paar{N(u_1) - N(u_2)}{v} \le \nu L \int_\Om \betrag{g(\cdot,t)} \betrag{u_1 - u_2} \betrag{v}.
\end{equation*}
Together with H{\"o}lder's inequality this proves the first inequality. The second and third one, follow from the first one and \eqref{E:L2Enorm}.
\end{proof}

\begin{BEM}
Using the continuous embedding of $H^1_0(\Om)$ into $L^p(\Om)$ with $p < \infty$ if $d = 2$ and $p = 6$ if $d = 3$, the terms $\nu L \la \ga(t)$ and $\nu L \la^2\ga(t)$ in Lem\-ma \ref{L:LipN} can be replaced by $\min \Bigl\{ \nu L \be^{-\halb} \ga(t), \nu L \cemb \ep^{-\halb} \normL[q]{\Om}{g(\cdot,t)} \Bigr\}$ and $\min \Bigl\{ \nu L \be^{-1} \ga(t), \linebreak[4]\nu L \cemb^2 \ep^{-1} \normL[r]{\Om}{g(\cdot,t)} \Bigr\}$, resp. where $q = \frac{2p}{p - 2}$, $r = \frac{p}{p-2}$, and $\cemb = \sup_{v} \frac{\normL[p]{\Om}{v}}{\normL{\Om}{\nabla v}}$.
\end{BEM}

\subsection{Equivalence of residual and error}\label{S:equireserr} With the discrete solution $u_\CI$ we associate the residual $R(u_\CI) \in L^2(0,T;H^{-1})$ by setting for all $v \in H^1_0(\Om)$
\begin{equation*}
\paar{R(u_\CI)}{v} = \paar{N(u_\CI)}{v} - \paar{\partial_t u_\CI}{v} - B(u_\CI,v).
\end{equation*}
Notice, that $B$ and $N$ are given by \eqref{E:defB-N} and that $\partial_t u_\CI = \frac{1}{\tau_n} \left( u^n_{\CT_n} - u^{n-1}_{\CT_{n-1}} \right)$ on $[t_{n-1},t_n]$. With this notation, we have the following equivalence of error and residual.

\begin{LEM}[Equivalence of error and residual]\label{L:equireserr}
For all $1 \le n \le N_\CI$ the $L^2(t_{n-1},\linebreak[4]t_n;H^{-1}(\Om))$-norm of the residual is bounded from above by the $X(t_{n-1},t_n)$-norm of the error
\begin{equation*}
\begin{split}
&\normXD[t_{n-1}]{t_n}{R(u_\CI)} \\
&\quad\le \norm{u - u_\CI}_{X(t_{n-1},t_n)} \sqrt{2} \creac \left\{ 1 + \nu L \la \min \left\{ \la, \sqrt{\tau_n} \right\} \max_{t_{n-1} \le t \le t_n}\ga(t) \right\}.
\end{split}
\end{equation*}
Conversely, the $X(0,T)$-norm of the error is bounded from above by the $L^2(0,T;\linebreak[4]H^{-1}(\Om))$-norm of the residual
\begin{equation*}
\begin{split}
\normX{u - u_\CI} 
&\le \left\{ \normL{\Om}{u_0 - u^0_{\CT_0}}^2 + \normXD{T}{R(u_\CI)}^2 \right\}^\halb \cdot \\
&\quad\;\biggl\{ 3 + \Bigl[1 + 3 \max \bigl\{ \creac^2, \nu^2 L^2 \la^2 \ga^2 \min \{ T, \la^2 \} \bigr\} \Bigr] e^{2 \nu L \ga T} \biggr\}^\halb.
\end{split}
\end{equation*}
If in addition $\ka = 2 \nu L \min \{ T, \la^2 \} \ga < 1$, the upper bound for the norm of the error can be improved to
\begin{equation*}
\begin{split}
\normX{u - u_\CI}
&\le \left\{ \normL{\Om}{u_0 - u^0_{\CT_0}}^2 + \normXD{T}{R(u_\CI)}^2 \right\}^\halb \cdot \\
&\quad\;\left\{ 3 + \left[1 + 3 \max \left\{ \creac^2, \halb \nu L \la^2 \ga \right\} \right] \frac{1}{1 - \ka}\right\}^\halb.
\end{split}
\end{equation*}
\end{LEM}

\begin{proof}
The variational formulation \eqref{E:varprob} and the definition of the residual yield
\begin{equation}\label{E:represerr}
\paar{\partial_t (u - u_\CI)}{v} + B(u - u_\CI,v) = \paar{N(u) - N(u_\CI)}{v} + \paar{R(u_\CI)}{v}
\end{equation} 
for all $v \in H^1_0(\Om)$ and almost all $t \in (0,T)$. Therefore, \cite[Proposition 6.14]{Ver13} and the assumption $\creac \ge 1$ imply for all $1 \le n \le N_\CI$
\begin{equation*}
\begin{split}
&\normXD[t_{n-1}]{t_n}{R(u_\CI)} \\
&\quad\le \sqrt{2} \creac \left\{ \norm{u - u_\CI}_{X(t_{n-1},t_n)} + \normXD[t_{n-1}]{t_n}{N(u) - N(u_\CI)} \right\}.
\end{split}
\end{equation*}
Together with Lemma \ref{L:LipN} this proves the upper bound for the dual norm of the residual. \\
To prove the upper bounds for the error, we go back to the proof of \cite[Proposition 6.14]{Ver13} and first observe that
\begin{equation*}
\begin{split}
&\paar{\partial_t (u - u_\CI) + \cd (u - u_\CI)}{v} \\
&\quad= \int_\Om \bigl[ \ep \nabla (u_\CI - u) \cdot \nabla v + b (u_\CI - u) v \bigr] + \paar{N(u) - N(u_\CI)}{v} + \paar{R(u_\CI)}{v},
\end{split}
\end{equation*}
Together with Lemma \ref{L:LipN} this implies
\begin{equation*}
\normD{\partial_t (u - u_\CI) + \cd (u - u_\CI)} \le \normD{R(u_\CI)} + \creac \normE{u - u_\CI} + \nu L \la \ga(t) \normL{\Om}{u - u_\CI}
\end{equation*}
and
\begin{equation*}
\begin{split}
&\int_0^T \normD{\partial_t (u - u_\CI) + \cd (u - u_\CI)}^2 \\
&\quad\le 3 \left\{ \int_0^T \normD{R(u_\CI)}^2 + \creac^2 \int_0^T \normE{u - u_\CI}^2 \right. \\
&\quad\quad\quad\quad\left. + \nu^2 L^2 \la^2 \ga^2 \min \left\{ T \sup_{0 < t < T} \normL{\Om}{u - u_\CI}^2, \la^2 \int_0^T \normE{u - u_\CI}^2 \right\} \right\}.
\end{split}
\end{equation*}
In order to bound $\sup_{0 < t < T} \normL{\Om}{u - u_\CI}^2$ and $\int_0^T \normE{u - u_\CI}^2$, we now use a standard parabolic energy argument and insert $u - u_\CI$ as test-function $v$ in \eqref{E:represerr}. Thanks to the coercivity of the bilinear form $B$ and Lemma \ref{L:LipN} this yields
\begin{equation*}
\begin{split}
&\halb \frac{\dx}{\dx t} \normL{\Om}{u - u_\CI}^2 + \normE{u - u_\CI}^2 \\
&\quad\le \halb \frac{\dx}{\dx t} \normL{\Om}{u - u_\CI}^2 + B(u - u_\CI, u - u_\CI) \\
&\quad= \paar{N(u) - N(u_\CI)}{u - u_\CI} + \paar{R(u_\CI)}{u - u_\CI} \\
&\quad\le \nu L \ga(t) \normL{\Om}{u - u_\CI}^2 + \normD{R(u_\CI)} \normE{u - u_\CI} \\
&\quad\le \nu L \ga(t) \normL{\Om}{u - u_\CI}^2 + \halb \normD{R(u_\CI)}^2 + \halb \normE{u - u_\CI}^2
\end{split}
\end{equation*}
and thus
\begin{equation*}
\begin{split}
&\normL{\Om}{(u - u_\CI)(\cdot,t)}^2 + \int_0^t \normE{u - u_\CI}^2 \\
&\quad\le 2 \nu L \ga \int_0^t \normL{\Om}{u - u_\CI}^2 + \int_0^t \normD{R(u_\CI)}^2 + \normL{\Om}{u_0 - u^0_{\CT_0}}^2.
\end{split}
\end{equation*}
If $\ka < 1$ we may absorb the first term on the right-hand side of this estimate by the left-hand side and obtain
\begin{equation*}
\begin{split}
&\sup_{0 < t < T} \normL{\Om}{u - u_\CI}^2 + \int_0^T \normE{u - u_\CI}^2 \\
&\quad\le \frac{1}{1 - \ka} \left\{ \normL{\Om}{u_0 - u^0_{\CT_0}}^2 + \normXD{T}{R(u_\CI)}^2 \right\}.
\end{split}
\end{equation*}
Otherwise, Gronwall's Lemma yields
\begin{equation*}
\begin{split}
&\sup_{0 < t < T} \normL{\Om}{u - u_\CI}^2 + \int_0^T \normE{u - u_\CI}^2 \\
&\quad\le e^{2 \nu L \ga T} \left\{ \normL{\Om}{u_0 - u^0_{\CT_0}}^2 + \normXD{T}{R(u_\CI)}^2 \right\}.
\end{split}
\end{equation*}
Combining these estimates with the bound for $\int_0^T \normD{\partial_t (u - u_\CI) + \cd (u - u_\CI)}^2$ establishes the upper bound for the error.
\end{proof}

\subsection{Decomposition of the residual}\label{S:decompR} We additively split the residual
\begin{equation*}
R(u_\CI) = R_\tau(u_\CI) + R_h(u_\CI) + R_D(u_\CI)
\end{equation*}
into a temporal residual, a spatial residual, and a data residual which, for all $v \in H^1_0(\Om)$, are defined by
\begin{equation*}
\begin{split}
\paar{R_\tau(u_\CI)}{v} &= \paar{N^{n\vartheta}(u_\CI)}{v} - \paar{N^{n\vartheta}(U^{n\vartheta})}{v} + B^{n\theta}(U^{n\theta} - u_\CI,v), \\
\paar{R_h(u_\CI)}{v} &= \paar{N^{n\vartheta}(U^{n\vartheta})}{v} - \paar{\partial_t u_\CI}{v} - B^{n\theta}(U^{n\theta},v), \\
\paar{R_D(u_\CI)}{v} &= \paar{N(u_\CI)}{v} - \paar{N^{n\vartheta}(u_\CI)}{v} - B(u_\CI,v) + B^{n\theta}(u_\CI,v).
\end{split}
\end{equation*}
In addition, we additively split the temporal residual
\begin{equation*}
R_\tau(u_\CI) = R_{\tau,\text{lin}}(u_\CI) + R_{\tau,\text{nonlin}}(u_\CI)
\end{equation*}
into a linear and a non-linear part which, for all $v \in H^1_0(\Om)$, are defined by
\begin{equation*}
\begin{split}
\paar{R_{\tau,\text{lin}}(u_\CI)}{v} &= B^{n\theta}(U^{n\theta} - u_\CI,v) \\
\paar{R_{\tau,\text{nonlin}}(u_\CI)}{v} &= \paar{N^{n\vartheta}(u_\CI)}{v} - \paar{N^{n\vartheta}(U^{n\vartheta})}{v}.
\end{split}
\end{equation*}
In the following subsections we will estimate the three residuals separately. The following Lemma shows that this is permissible. Lemma \ref{L:lintempresdom} below in addition shows that the temporal residual is dominated by its linear part if $\nu L \la^2 \ga$ is sufficiently small.

\begin{LEM}[Decomposition of the residual]\label{L:decompR}
For every $n \in \left\{ 1, \ldots, N_\CI \right\}$ we have
\begin{equation*}
\begin{split}
\normXD[t_{n-1}]{t_n}{R_\tau(u_\CI) + R_h(u_\CI)} &\le \normXD[t_{n-1}]{t_n}{R_{\tau,\text{lin}}(u_\CI)} \\
&\quad+ \normXD[t_{n-1}]{t_n}{R_{\tau,\text{nonlin}}(u_\CI)} \\
&\quad+ \normXD[t_{n-1}]{t_n}{R_h(u_\CI)}
\end{split}
\end{equation*}
and
\begin{equation*}
\begin{split}
&\frac{2}{25} \left\{ \normXD[t_{n-1}]{t_n}{R_{\tau,\text{lin}}(u_\CI)}^2 + \normXD[t_{n-1}]{t_n}{R_h(u_\CI)}^2 \right\}^\halb \\
&\quad\le \normXD[t_{n-1}]{t_n}{R_\tau(u_\CI) + R_h(u_\CI)} + \normXD[t_{n-1}]{t_n}{R_{\tau,\text{nonlin}}(u_\CI)}.
\end{split}
\end{equation*}
\end{LEM}

\begin{proof}
Since $\sqrt{\frac{5}{14}} \left(1 - \frac{\sqrt{3}}{2} \right) > \frac{2}{25}$ and $R_{\tau,\text{lin}}$ is affine in $U^{n\theta} - u_\CI$ and thus proportional to $\frac{t - t_{n-1}}{\tau_n} - \theta$, the estimates follow from the triangle inequality and \cite[Lemma 6.16]{Ver13}.
\end{proof}

\subsection{Bounding the data residual}\label{S:datares} H{\"o}lder's inequality and \eqref{E:L2Enorm} yield the following upper bound for the data residual.

\begin{LEM}[Upper bound for the data residual]\label{L:datares}
For every $n \in \left\{ 1, \ldots, N_\CI \right\}$ we have
\begin{equation*}
\begin{split}
&\normXD[t_{n-1}]{t_n}{R_D(u_\CI)} \\
&\quad \le \nu L \la \left\{ \norm{g - g^{n\vartheta}}_{L^2(t_{n-1},t_n;L^2(\Om))} \phantom{\left( \int_{t_{n-1}}^{t_n} \normE{u_\CI}^2 \right)^\halb}\right.\\
&\quad\quad\quad\quad\quad\left.+ \norm{g - g^{n\vartheta}}_{L^\infty(t_{n-1},t_n;L^\infty(\Om))} \left( \int_{t_{n-1}}^{t_n} \normE{u_\CI}^2 \right)^\halb \right\}  \\
&\quad\quad+ \ep^{-\halb} \la \norm{\Va - \Va^{n\theta}}_{L^\infty(t_{n-1},t_n;L^\infty(\Om))}  \left( \int_{t_{n-1}}^{t_n} \normE{u_\CI}^2 \right)^\halb \\
&\quad\quad+ \la^2 \norm{b - b^{n\theta}}_{L^\infty(t_{n-1},t_n;L^\infty(\Om))}  \left( \int_{t_{n-1}}^{t_n} \normE{u_\CI}^2 \right)^\halb.
\end{split}
\end{equation*}
\end{LEM}

\begin{BEM}
Since $u_\CI = \frac{t - t_{n-1}}{\tau_n} u^n_{\CT_n} + \frac{t_n - t}{\tau_n} u^{n-1}_{\CT_{n-1}}$ for $t_{n-1} \le t \le t_n$, the convexity of $\normE{\cdot}^2$ and Simpson's rule yield
\begin{equation*}
\int_{t_{n-1}}^{t_n} \normE{u_\CI}^2 \le \frac{\tau_n}{2} \left( \normE{u^n_{\CT_n}}^2 + \normE{u^{n-1}_{\CT_{n-1}}}^2 \right).
\end{equation*}
\end{BEM}

\subsection{Bounding the temporal residual}\label{S:tempres} We first bound the linear part of the temporal residual. \\
For every time-interval $[t_{n-1},t_n]$ we have
\begin{equation*}
R_{\tau,\text{lin}} (u_\CI) = \left( \theta - \frac{t - t_{n-1}}{\tau_n} \right) r^n
\end{equation*}
where $r^n \in H^{-1}(\Om)$ is defined by
\begin{equation*}
\paar{r^n}{v} = B^{n\theta}(u^n_{\CT_n} - u^{n-1}_{\CT_{n-1}},v)
\end{equation*}
for $v \in H^1_0(\Om)$.  The assumption $\creac \ge 1$ and \cite[Lemma 6.17]{Ver13} therefore yield the following upper and lower bounds for the linear part of the temporal residual.

\begin{LEM}[Bounds for the linear part of the temporal residual]\label{L:lintempres}
For every $n \in \left\{ 1, \ldots, N_\CI \right\}$, the linear part of the temporal residual can be bounded from above and from below by
\begin{equation*}
\begin{split}
&\frac{\sqrt{\tau_n}}{\sqrt{12} (2 + \creac)} \left\{ \normE{u^n_{\CT_n} - u^{n-1}_{\CT_{n-1}}} + \normD{\Va^{n\theta} \cdot \nabla (u^n_{\CT_n} - u^{n-1}_{\CT_{n-1}})} \right\} \\
&\quad\le \normXD[t_{n-1}]{t_n}{R_{\tau,\text{lin}}(u_\CI)} \\
&\quad\quad\le \frac{\sqrt{\tau_n}}{\sqrt{3} \creac} \left\{ \normE{u^n_{\CT_n} - u^{n-1}_{\CT_{n-1}}} + \normD{\Va^{n\theta} \cdot \nabla (u^n_{\CT_n} - u^{n-1}_{\CT_{n-1}})} \right\}.
\end{split}
\end{equation*}
\end{LEM}

The term $\normD{\Va^{n\theta} \cdot \nabla (u^n_{\CT_n} - u^{n-1}_{\CT_{n-1}})}$ is not suited for a posteriori error estimates since it involves the dual norm $\normD{\cdot}$. The next two Lemmas bound this term for the case of dominant diffusion, i.e. $\ep \gtrsim 1$, and of dominant convection, i.e. $\ep \ll 1$, respectively. The first one follows from H{\"o}lder's inequality and \eqref{E:Fried}, the second one from \cite[Lemma 6.18]{Ver13}.

\begin{LEM}[Bounding the convective derivative for dominant diffusion]\label{L:convder-domdiff}
For every $n \in \left\{ 1, \ldots, N_\CI \right\}$ we have
\begin{equation*}
\normD{\Va^{n\theta} \cdot \nabla (u^n_{\CT_n} - u^{n-1}_{\CT_{n-1}})} \le \ep^{-\halb} \la \normL[\infty]{\Om}{\Va^{n\theta}} \normE{u^n_{\CT_n} - u^{n-1}_{\CT_{n-1}}}.
\end{equation*}
\end{LEM}

\begin{LEM}[Bounding the convective derivative for dominant convection]\label{L:convder-domconv}
\hfill For \linebreak[4]every $n \in \left\{ 1, \ldots, N_\CI \right\}$ denote by $S^{1,0}_0(\widetilde \CT_n)$ the space of continuous, piecewise affine functions vanishing on $\Ga$ corresponding to the partition $\widetilde \CT_n$ and by $\widetilde u^{n}_{\CT_n} \in S^{1,0}_0(\widetilde \CT_n)$ the unique solution of the discrete reaction-diffusion problem
\begin{equation*}
\ep \int_\Om \nabla \widetilde u^{n}_{\CT_n} \cdot \nabla v_{\CT_n} + \be \int_\Om \widetilde u^{n}_{\CT_n}
v_{\CT_n} = \int_\Om \Va^{n\theta} \cdot \nabla (u^n_{\CT_n} - u^{n-1}_{\CT_{n-1}}) v_{\CT_n}
\end{equation*}
for all $v_{\CT_n} \in S^{1,0}_0(\widetilde \CT_n)$. Define the error indicator $\widetilde \eta^{n}_{\CT_n}$ by
\begin{equation*}
\begin{split}
\widetilde \eta^{n}_{\CT_n} &= \left\{ \sum_{K \in \widetilde \CT_n}
\mdiam{K}^2 \normL{K}{\Va^{n\theta} \cdot \nabla (u^n_{\CT_n} - u^{n-1}_{\CT_{n-1}}) + \ep \De \widetilde u^{n}_{\CT_n} - \be \widetilde u^{n}_{\CT_n}}^2 \right. \\
&\quad\quad \left. + \sum_{E \in \widetilde \CE_{n,\Om}} \ep^{- \halb} \mdiam{E}
\normL{E}{\jump{\Vn_E \cdot \nabla \widetilde u^{n}_{\CT_n}}}^2 \right\}^\halb,
\end{split}
\end{equation*}
and the data error $\widetilde \theta^{n}_{\CT_n}$ by
\begin{equation*}
\widetilde \theta^{n}_{\CT_n} = \left\{ \sum_{K \in \widetilde \CT_n} \mdiam{K}^2 \normL{K}{(\Va^{n\theta} - \Va^{n\theta}_{\widetilde \CT_n}) \cdot \nabla (u^n_{\CT_n} - u^{n-1}_{\CT_{n-1}})}^2 \right\}^\halb
\end{equation*}
where $\mdiam{\om} = \min \left\{ \ep^{- \halb} \diam(\om), \be^{- \halb} \right\}$ and $\Va^{n\theta}_{\widetilde \CT_n}$ is an approximation of $\Va^{n\theta}$ on $\widetilde \CT_n$. Then there are two constants $c_\dagger$ and $c^\dagger$ which only depend on the shape-parameters $\shape$ and $\shape[\widetilde \CT,\CT]$ such that the following estimates are valid
\begin{equation*}
c_\dagger \left\{ \normE{\widetilde u^{n}_{\CT_n}} + \widetilde \eta^{n}_{\CT_n} - \widetilde \theta^{n}_{\CT_n}\right\} \le \normD{\Va^{n\theta} \cdot \nabla (u^n_{\CT_n} - u^{n-1}_{\CT_{n-1}})} \le c^\dagger
\left\{ \normE{\widetilde u^{n}_{\CT_n}} + \widetilde \eta^{n}_{\CT_n} \right\}.
\end{equation*}
\end{LEM}

Next we bound the non-linear part of the temporal residual.

\begin{LEM}[Upper bounds for the non-linear temporal residual]\label{L:nonlintempres}
For every $n \linebreak[4]\in \left\{ 1, \ldots, N_\CI \right\}$, the non-linear part of the temporal residual can be bounded from above by
\begin{equation*}
\begin{split}
\normXD[t_{n-1}]{t_n}{R_{\tau,\text{nonlin}}(u_\CI)}
&\le \sqrt{\frac{\tau_n}{3}} \nu L \la \ga \normL{\Om}{u^n_{\CT_n} - u^{n-1}_{\CT_{n-1}}} \\
&\le \sqrt{\frac{\tau_n}{3}} \nu L \la^2 \ga \normE{u^n_{\CT_n} - u^{n-1}_{\CT_{n-1}}}.
\end{split}
\end{equation*}
\end{LEM}

\begin{proof}
The assertion follows from \eqref{E:L2Enorm}, Lemma \ref{L:LipN},
\begin{equation*}
\int_{t_{n-1}}^{t_n} \normL{\Om}{u_\CI - U^{n\vartheta}}^2 \le \normL{\Om}{u^n_{\CT_n} - u^{n-1}_{\CT_{n-1}}}^2 \int_{t_{n-1}}^{t_n} \left( \vartheta - \frac{t - t_{n-1}}{\tau_n} \right)^2,
\end{equation*}
and
\begin{equation*}
\int_{t_{n-1}}^{t_n} \left( \vartheta - \frac{t - t_{n-1}}{\tau_n} \right)^2 = \frac{\tau_n}{6} \left[ 2 - 6 \vartheta(1 - \vartheta) \right] \le \frac{\tau_n}{3}.
\end{equation*}
\end{proof}

Lemma \ref{L:nonlintempres} and the estimate
\begin{equation*}
\normL{\Om}{u^n_{\CT_n} - u^{n-1}_{\CT_{n-1}}} \le 2 \sup_{t_{n-1} \le t \le t_n} \normL{\Om}{(u - u_\CI)(\cdot,t)} + \sqrt{\tau_n} \normL{\Om \times (t_{n-1},t_n)}{\partial_t u}
\end{equation*}
yield the following upper bound for the non-linear part of the temporal residual for all parameters $\ep$, $\be$, $\nu$, and $\ga$.

\begin{LEM}[Non-linear temporal residual and error]\label{L:nonlintempreserr}
For all parameters $\ep$, $\be$, $\nu$, and $\ga$ the non-linear part of the temporal residual is bounded from above by the error and the $L^2$-norm of $\partial_t u$,  i.e. for every $n \in \left\{ 1, \ldots, N_\CI \right\}$ we have
\begin{equation*}
\begin{split}
\normXD[t_{n-1}]{t_n}{R_{\tau,\text{nonlin}}(u_\CI)} &\le \frac{2 \sqrt{\tau_n}}{\sqrt{3}} \nu L \la \ga \sup_{t_{n-1} \le t \le t_n} \normL{\Om}{(u - u_\CI)(\cdot,t)} \\
&\quad+ \frac{\tau_n}{\sqrt{3}} \nu L \la \ga \normL{\Om \times (t_{n-1},t_n)}{\partial_t u}.
\end{split}
\end{equation*}
\end{LEM}

If, on the other hand, $\nu L \la^2 \ga$ is sufficiently small, Lemmas \ref{L:decompR}, \ref{L:lintempres}, and \ref{L:nonlintempres} imply that the temporal residual is dominated by its linear part.

\begin{LEM}[Non-linear and linear temporal residual]\label{L:lintempresdom}
If $\widetilde \ka = 25 \left( 2 + \creac \right) \nu L \la^2 \ga \linebreak[4]< 1$, the temporal residual is dominated by its linear part, i.e. for every $n \in \left\{ 1, \ldots, N_\CI \right\}$ we have
\begin{equation*}
\begin{split}
&\frac{2}{25} (1 - \widetilde \ka) \left\{ \normXD[t_{n-1}]{t_n}{R_{\tau,\text{lin}}(u_\CI)}^2 + \normXD[t_{n-1}]{t_n}{R_h(u_\CI)}^2 \right\}^\halb \\
&\quad\le \normXD[t_{n-1}]{t_n}{R_\tau(u_\CI) + R_h(u_\CI)}.
\end{split}
\end{equation*}
\end{LEM}

\subsection{Bounding the spatial residual}\label{S:spatres} Replacing in \cite[Lemma 3.5]{TobVer15}  the right-hand side $f^{n\theta}$ by $\nu \vfi(U^{n\theta}) g^{n\theta}$ yields the following bounds for the spatial residual.

\begin{LEM}[Bounds for the spatial residual]\label{L:spatres}
For every $n \in \{ 1, \ldots, N_\CI \}$ define a spatial error indicator by
\begin{equation*}
\begin{split}
\eta^n_{\CT_n} &= \left\{ \sum_{K \in \widetilde \CT_n} \mdiam{K}^2 \left\lVert \nu \vfi(\overline U^{n\theta}_{\CT_n}) g^{n\theta}_{\CT_n} - \frac{1}{\tau_n} \left( u^n_{\CT_n} - u^{n-1}_{\CT_{n-1}} \right) + \ep \Delta U^{n\theta} \right.\right. \\
&\quad\quad\left.\phantom{\frac{u^n_{\CT_n} - u^{n-1}_{\CT_{n-1}}}{\tau_n}} - \Va^{n\theta}_{\CT_n} \cdot \nabla U^{n\theta} - b^{n\theta}_{\CT_n} U^{n\theta} \right\rVert_K^2 \\
&\quad\quad\quad \left. + \halb \sum_{E \in \widetilde \CE_{n,\Om}} \ep^{- \halb} \mdiam{E} \normL{E}{\jump{\ep \Vn_E \cdot \nabla U^{n\theta}}}^2 \right\}^\halb
\end{split}
\end{equation*}
and spatial data errors by
\begin{equation*}
\begin{split}
\theta^n_{\CT_n} &= \left\{ \sum_{K \in \CT_n} \mdiam{K}^2 \Bigl\lVert \nu \vfi(U^{n\theta}) \left( g^{n\theta}_{\CT_n} - g^{n\theta} \right) + \nu (\vfi(U^{n\theta}) - \vfi(\overline U^{n\theta}_{\CT_n})) g^{n\theta}_{\CT_n} \right. \\
&\quad\quad\quad\quad\left.\phantom{\sum_{K \in \CT_n}} + (\Va^{n\theta}_{\CT_n} - \Va^{n\theta}) \cdot \nabla U^{n\theta} + (b^{n\theta}_{\CT_n} - b^{n\theta}) U^{n\theta} \Bigr\rVert_K^2 \right\}^\halb, \\
\Theta^n_{\mathop{cip},\CT_n} &= \left\{ \sum_{K \in \CT_n} \mdiam{K}^2 \normL{K}{\left( \Va^{n\theta} - \Va^{n\theta}_{\CT_n} \right) \cdot \nabla U^{n\theta}}^2 + \mdiam{K}^2 \hdiamK^2 \normL[\infty]{K}{\nabla \Va^{n\theta}} \normL{K}{\nabla U^{n\theta}}^2 \right\}^\halb.
\end{split}
\end{equation*}
Here, $U^{n\theta} = \theta u^n_{\CT_n} + (1 - \theta) u^{n-1}_{\CT_{n-1}}$ is as in \eqref{E:discprob}, $\overline U^{n\theta}_{\CT_n}$ is a piecewise constant approximation of  $U^{n\theta}$ on $\CT_n$, $g^{n\theta}$, $\Va^{n\theta}$, and $b^{n\theta}$ are as in \eqref{E:ntheta}, and $g^{n\theta}_{\CT_n}$, $\Va^{n\theta}_{\CT_n}$, and $b^{n\theta}_{\CT_n}$ are approximations of $g^{n\theta}$, $\Va^{n\theta}$, and $b^{n\theta}$ on $\CT_n$. Then, on every interval $(t_{n-1},t_n]$, the dual norm of the spatial residual can be bounded from above by
\begin{equation*}
\normD{R_h(u_\CI)} \le c^\flat \left\{ \left( \eta^n_{\CT_n} \right)^2 + \left( \theta^n_{\CT_n} \right)^2 
%+ \si_{\mathop{lps}} \left( \Theta^n_{\mathop{lps},\CT_n} \right)^2 
+ \si_{\mathop{cip}} \left( \Theta^n_{\mathop{cip},\CT_n} \right)^2 \right\}^\halb
\end{equation*}
and from below by
\begin{equation*}
\eta^n_{\CT_n} \le c_\flat \left[ \normD{R_h(u_\CI)} + \theta^n_{\CT_n} \right].
\end{equation*}
Here, the parameter $\si_{\mathop{cip}}$ equals $1$ for the continuous interior penalty method and vanishes for the other stabilizations. The above error estimates are robust in the sense that the constants $c^\flat$ and $c_\flat$ are independent of the parameters $\ep$, $\be$, and $\nu$.
\end{LEM}

\subsection{A posteriori error estimates}\label{S:aposterr} Lemmas \ref{L:equireserr}, \ref{L:decompR}, \ref{L:datares}, \ref{L:lintempres}, \ref{L:convder-domdiff}, \ref{L:convder-domconv}, \ref{L:nonlintempreserr}, \ref{L:lintempresdom}, and \ref{L:spatres} yield the following a posteriori error estimates.

\begin{THM}[A posteriori error estimates]\label{T:apost}
The error between the solution $u$ of problem \eqref{E:varprob} and the solution $u_\CI$ of problem \eqref{E:discprob} is bounded from above by
\begin{equation*}
\begin{split}
&\left\{ \sup_{0 < t < T} \normL[\infty]{\Om}{u - u_\CI}^2 + \int_0^T \normE{u - u_\CI}^2 + \int_0^T \normD{\partial_t (u - u_\CI) + \cd (u - u_\CI)}^2 \right\}^\halb \\
&\quad\quad\le c^\ast \left\{ \normL{\Om}{u_0 - \pi_0 u_0}^2 \phantom{\sum_{n = 1}^{N_\CI}} \right. \\
&\quad\quad\quad\phantom{\sum_{n = 1}^{N_\CI}}
+ \sum_{n = 1}^{N_\CI} \tau_n \left[ \left(\eta^n_{\CT_n}\right)^2 + \normE{u^n_{\CT_n} - u^{n-1}_{\CT_{n-1}}}^2 + \left(\widetilde \eta^{n}_{\CT_n}\right)^2 + \normE{\widetilde u^{n}_{\CT_n}}^2 + \left(\widetilde \theta^{n}_{\CT_n}\right)^2 \right] \\
&\quad\quad\quad\phantom{c^\ast \Bigl\{}
+ \sum_{n = 1}^{N_\CI} \tau_n \left[ \left(\theta^n_{\CT_n}\right)^2 
+ \si_{\mathop{cip}} \left( \Theta^n_{\mathop{cip},\CT_n} \right)^2 \right] \\
&\quad\quad\quad\quad\;\; + \norm{g - g^{n\vartheta}}_{L^\infty(0,T;L^\infty(\Om))}^2 \left( 1 + \int_0^T \normE{u_\CI}^2 \right) \\
&\quad\quad\quad\quad\;\;\left. + \left( \norm{\Va - \Va^{n\theta}}_{L^\infty(0,T;L^\infty(\Om))}^2 + \norm{b - b^{n\theta}}_{L^\infty(0,T;L^\infty(\Om))}^2 \right) \int_0^T \normE{u_\CI}^2 \right\}^\halb
\end{split}
\end{equation*}
and on each interval $(t_{n-1},t_n]$, $1 \le n \le N_\CI$, from below by
\begin{equation*}
\begin{split}
&\tau_n^\halb \left\{\left(\eta^n_{\CT_n}\right)^2 + \normE{u^n_{\CT_n} - u^{n-1}_{\CT_{n-1}}}^2 + \left(\widetilde \eta^{n}_{\CT_n}\right)^2 + \normE{\widetilde u^{n}_{\CT_n}}^2 \right\}^\halb \\
&\quad\le c_\ast \left\{ \sup_{t_{n-1} \le t \le t_n} \normL{\Om}{u - u_\CI}^2 + \int_{t_{n-1}}^{t_n} \normE{u - u_\CI}^2  \right. \\
&\quad\quad\quad\quad+ \int_{t_{n-1}}^{t_n} \normD{\partial_t (u - u_\CI) + \cd (u - u_\CI)}^2 \\ 
&\quad\quad\quad\quad+ \tau_n \left(\theta^n_{\CT_n}\right)^2 \\
&\quad\quad\quad\quad + \norm{g - g^{n\vartheta}}_{L^\infty(t_{n-1},t_n;L^\infty(\Om))}^2 \left( 1 + \int_{t_{n-1}}^{t_n} \normE{u_\CI}^2 \right) \\
&\quad\quad\quad\quad + \left( \norm{\Va - \Va^{n\theta}}_{L^\infty(t_{n-1},t_n;L^\infty(\Om))}^2 \right. \\
&\quad\quad\quad\quad\quad\quad\left.\left. + \norm{b - b^{n\theta}}_{L^\infty(t_{n-1},t_n;L^\infty(\Om))}^2 \right) \int_{t_{n-1}}^{t_n} \normE{u_\CI}^2 \right\}^\halb \\
&\quad\quad + c_{\ast\ast} \left\{\tau_n^\halb \sup_{t_{n-1} \le t \le t_n} \normL{\Om}{u - u_\CI} + \tau_n \normL{\Om \times (t_{n-1},t_n)}{\partial_t} \right\}.
\end{split}
\end{equation*}
Here, the functions $\widetilde u^{n}_{\CT_n}$ and the indicators $\widetilde \eta^{n}_{\CT_n}$ and $\widetilde \theta^{n}_{\CT_n}$ are defined in Lemma \ref{L:convder-domconv}, and the quantities $\eta^n_{\CT_n}$, $\theta^n_{\CT_n}$,  and $\Theta^n_{\mathop{cip},\CT_n}$ are as in Lemma \ref{L:spatres}. The functions $\widetilde u^{n}_{\CT_n}$ and the indicators $\widetilde \eta^{n}_{\CT_n}$ and $\widetilde \theta^{n}_{\CT_n}$ may be dropped if $\ep \gtrsim 1$. The parameter $\si_{\mathop{cip}}$ equals $1$ for the continuous interior penalty scheme and vanishes for the other stabilizations. For arbitrary parameters $\ep$, $\be$, $\nu$, and $\ga$, the constant $c^\ast$ is proportional to $\nu L \la^2 \ga$ and $e^{\nu L \ga T}$ with factors depending on the shape parameters $\shape$ and $\shape[\widetilde \CT,\CT]$, the constant $c_\ast$ is proportional to $\nu L \la^2 \ga$ with factors depending on the shape parameters $\shape$ and $\shape[\widetilde \CT,\CT]$ and the polynomial degrees of the finite element functions, and the constant $c_{\ast\ast}$ is proportional to $\nu L \la \ga$. If $\ka = 2 \nu L \min \{ T, \la^2 \} \ga < 1$, the constant $c^\ast$only depends on $\ka$ and the shape parameters $\shape$ and $\shape[\widetilde \CT,\CT]$. If in addition $\widetilde \ka = 25 \left( 2 + \creac \right) \nu L \la^2 \ga < 1$, the constant $c_\ast$ only depends on $\widetilde \ka$, the shape parameters $\shape$ and $\shape[\widetilde \CT,\CT]$ and the polynomial degrees of the finite element functions and the constant $c_{\ast\ast}$ vanishes.
\end{THM}
%
% Acknowledgement
%
\section*{Acknowledgements} Our sincere thanks are due to A.~Prohl for drawing our attention to the subject and for fruitful discussions.
%
% References
%
%\bibliography{../../myLaTex/myReferences}

\begin{thebibliography}{10}

\bibitem{AmrWih}
M.~Amrein and T.~P. Wihler, \emph{An adaptive space-time {N}ewton--{G}alerkin
  approach for semilinear singularly perturbed parabolic evolution equations},
  IMA J. Numer. Anal. \textbf{37} (2017), no.~4, 2004--2019.

\bibitem{BanBrzNekPro}
{\v L}.~Ba{\v n}as, Z.~Brze{\'z}niak, M.~Neklyudov, and A.~Prohl,
  \emph{Stochastic ferromagnetism}, De Gruyter Studies in Mathematics, vol.~58,
  De Gruyter, Berlin, 2014, Analysis and numerics.

\bibitem{Bur05}
E.~Burman, \emph{A unified analysis for conforming and nonconforming stabilized
  finite element methods using interior penalty}, SIAM J. Numer. Anal.
  \textbf{43} (2005), no.~5, 2012--2033. \MR{2192329}

\bibitem{BurErn07}
E.~Burman and A.~Ern, \emph{Continuous interior penalty {$hp$}-finite element
  methods for advection and advection-diffusion equations}, Math. Comp.
  \textbf{76} (2007), no.~259, 1119--1140. \MR{2299768}

\bibitem{BurHan04}
E.~Burman and P.~Hansbo, \emph{Edge stabilization for {G}alerkin approximations
  of convection-diffusion-reaction problems}, Comput. Methods Appl. Mech.
  Engrg. \textbf{193} (2004), no.~15-16, 1437--1453. \MR{2068903}

\bibitem{CarMulPro12}
E.~Carelli, A.~M\"uller, and A.~Prohl, \emph{Domain decomposition strategies
  for the stochastic heat equation}, Int. J. Comput. Math. \textbf{89} (2012),
  no.~18, 2517--2542.

\bibitem{DunHauPro12}
T.~Dunst, E.~Hausenblas, and A.~Prohl, \emph{Approximate {E}uler method for
  parabolic stochastic partial differential equations driven by space-time
  {L}\'evy noise}, SIAM J. Numer. Anal. \textbf{50} (2012), no.~6, 2873--2896.

\bibitem{DunPro16}
T.~Dunst and A.~Prohl, \emph{The forward-backward stochastic heat equation:
  numerical analysis and simulation}, SIAM J. Sci. Comput. \textbf{38} (2016),
  no.~5, A2725--A2755.

\bibitem{ElAErnBur07}
L.~El~Alaoui, A.~Ern, and E.~Burman, \emph{A priori and a posteriori analysis
  of non-conforming finite elements with face penalty for advection-diffusion
  equations}, IMA J. Numer. Anal. \textbf{27} (2007), no.~1, 151--171.
  \MR{2289275}

\bibitem{ErnGue04}
A.~Ern and J.-L. Guermond, \emph{Theory and practice of finite elements},
  Applied Mathematical Sciences, vol. 159, Springer-Verlag, New York, 2004.

\bibitem{ErnGue13}
\bysame, \emph{Weighting the edge stabilization}, SIAM J. Numer. Anal.
  \textbf{51} (2013), no.~3, 1655--1677. \MR{3062586}

\bibitem{GeoLakVir}
E.~H. Georgoulis, O.~Lakkis, and J.~M. Virtanen, \emph{A posteriori error
  control for discontinuous {G}alerkin methods for parabolic problems}, SIAM J.
  Numer. Anal. \textbf{49} (2011), no.~2, 427--458.

\bibitem{Gue99}
J.-L. Guermond, \emph{Stabilization of {G}alerkin approximations of transport
  equations by subgrid modeling}, M2AN Math. Model. Numer. Anal. \textbf{33}
  (1999), no.~6, 1293--1316. \MR{1736900}

\bibitem{Gue01}
\bysame, \emph{Subgrid stabilization of {G}alerkin approximations of linear
  monotone operators}, IMA J. Numer. Anal. \textbf{21} (2001), no.~1, 165--197.
  \MR{1812271}

\bibitem{HeTob12}
L.~He and L.~Tobiska, \emph{The two-level local projection stabilization as an
  enriched one-level approach}, Adv. Comput. Math. \textbf{36} (2012), no.~4,
  503--523. \MR{2912460}

\bibitem{HugBro79}
T.~J.~R. Hughes and A.~Brooks, \emph{A multidimensional upwind scheme with no
  crosswind diffusion}, Finite element methods for convection dominated flows
  ({P}apers, {W}inter {A}nn. {M}eeting {A}mer. {S}oc. {M}ech. {E}ngrs., {N}ew
  {Y}ork, 1979), AMD, vol.~34, Amer. Soc. Mech. Engrs. (ASME), New York, 1979,
  pp.~19--35. \MR{571681}

\bibitem{HutJen}
M.~Hutzenthaler and A.~Jentzen, \emph{Numerical approximations of stochastic
  differential equations with non-globally {L}ipschitz continuous
  coefficients}, Mem. Amer. Math. Soc. \textbf{236} (2015), no.~1112, v+99.

\bibitem{KnoTob11}
P.~Knobloch and L.~Tobiska, \emph{On the stability of finite-element
  discretizations of convection-diffusion-reaction equations}, IMA J. Numer.
  Anal. \textbf{31} (2011), no.~1, 147--164. \MR{2755940}

\bibitem{LiuRoe}
W.~Liu and M.~R\"ockner, \emph{Stochastic partial differential equations: an
  introduction}, Universitext, Springer, Cham, 2015.

\bibitem{RooStyTob08}
H.-G. Roos, M.~Stynes, and L.~Tobiska, \emph{Robust numerical methods for
  singularly perturbed differential equations}, second ed., Springer Series in
  Computational Mathematics, vol.~24, Springer-Verlag, Berlin, 2008,
  Convection-diffusion-reaction and flow problems.

\bibitem{TobVer15}
L.~Tobiska and R.~Verf{\"u}rth, \emph{Robust {\it a posteriori} error estimates
  for stabilized finite element methods}, IMA J. Numer. Anal. \textbf{35}
  (2015), no.~4, 1652--1671.

\bibitem{TobWin10}
L.~Tobiska and C.~Winkel, \emph{The two-level local projection stabilization as
  an enriched one-level approach. {A} one-dimensional study}, Int. J. Numer.
  Anal. Model. \textbf{7} (2010), no.~3, 520--534. \MR{2644288}

\bibitem{Ver98e}
R.~Verf{\"u}rth, \emph{A posteriori error estimates for nonlinear problems.
  {$L\sp r(0,T; L\sp \rho(\Omega))$}-error estimates for finite element
  discretizations of parabolic equations}, Math. Comp. \textbf{67} (1998),
  no.~224, 1335 -- 1360.

\bibitem{Ver98d}
\bysame, \emph{A posteriori error estimates for nonlinear problems: {$L\sp
  r(0,T;$} {$W\sp {1,\rho}(\Omega))$}-error estimates for finite element
  discretizations of parabolic equations}, Numer. Methods Partial Differential
  Equations \textbf{14} (1998), no.~4, 487 -- 518.

\bibitem{Ver13}
\bysame, \emph{A {P}osteriori {E}rror {E}stimation {T}echniques for {F}inite
  {E}lement {M}ethods}, Oxford University Press, Oxford, 2013.

\bibitem{WeiHutJenKru}
E.~Weinan, M.~Hutzenthaler, A.~Jentzen, and T.~Kruse, \emph{On multilevel
  picard numerical approximations for high-dimensional nonlinear parabolic
  partial differential equations and high-dimensional nonlinear backward
  stochastic differential equations}, arXiv:1708.03223, 2017.

\end{thebibliography}
%\bibliographystyle{amsplain}
\providecommand{\bysame}{\leavevmode\hbox to3em{\hrulefill}\thinspace}
\providecommand{\MR}{\relax\ifhmode\unskip\space\fi MR }
\providecommand{\MRhref}[2]{%
  \href{http://www.ams.org/mathscinet-getitem?mr=#1}{#2}
}
\providecommand{\href}[2]{#2}

\end{document}